\documentclass[12pt]{amsart}

\usepackage{amssymb}
\usepackage{amsthm}

\usepackage[usenames]{color}

\usepackage[margin=1in]{geometry}
\usepackage[foot]{amsaddr}

\usepackage{hyperref}

\theoremstyle{plain}
\newtheorem{proposition}{Proposition}[section]
\newtheorem{theorem}[proposition]{Theorem}
\newtheorem{lemma}[proposition]{Lemma}
\newtheorem{corollary}[proposition]{Corollary}
\theoremstyle{definition}

\newtheorem{definition}[proposition]{Definition}

\theoremstyle{remark}
\newtheorem{remark}[proposition]{Remark}

\DeclareMathOperator{\tr}{tr}
\DeclareMathOperator{\supp}{supp}

\DeclareMathOperator{\OO}{\mathcal{O}}

\DeclareMathOperator{\RR}{\mathbb{R}}

\DeclareMathOperator{\JJ}{\mathbb{J}}
\DeclareMathOperator{\NN}{\mathbb{N}}

\newcommand{\abs}[1]{\left|#1\right|}

\newcommand{\norm}[1]{\left\|#1\right\|}
\newcommand{\wt}[1]{\widetilde{#1}}

\begin{document}

\title{Boundaries of non-compact harmonic manifolds}
\author{Andrew M. Zimmer}\address{Department of Mathematics, University of Michigan, Ann Arbor, MI 48109.}
\email{aazimmer@umich.edu}
\date{\today}
\keywords{harmonic manifolds, geodesic flow, topologically transitive, boundaries, compactifications}

\begin{abstract}
In this paper we consider non-compact non-flat simply connected harmonic manifolds. In particular, we show that the Martin boundary and Busemann boundary coincide for such manifolds. For any finite volume quotient we show that (up to scaling) there is a unique Patterson-Sullivan measure and this measure coincides with the harmonic measure. As an application of these results we prove that the geodesic flow on a non-flat finite volume harmonic manifold without conjugate points is topologically transitive. 
\end{abstract}

\maketitle

\section{Introduction}

A complete Riemannian manifold $M$ is called \emph{harmonic} if about any point the geodesic spheres of sufficiently small radii are of constant mean curvature. Examples of harmonic manifolds include rank one locally symmetric spaces and flat spaces. In this paper we study non-compact simply connected harmonic manifolds. It is well known that these manifolds have no conjugate points. 

Suppose $X$ is a simply connected non-compact harmonic manifold. Delaying definitions until Section 2, let $\partial \hat{X}$ denote the Busemann boundary of $X$ and $\partial_{\Delta} X$ denote the Martin boundary of $X$. For a unit tangent vector $v \in SX$ there is a natural associated function
\begin{align*}
b_v(x) = \lim_{t \rightarrow \infty} d(x,\gamma_v(t))-t
\end{align*}
where $\gamma_v$ is the geodesic in $X$ with $\gamma_v^\prime(0)=v$. As in~\cite{Led10, LW10} we normalize our Busemann functions such that $\xi(o)=0$ for a fixed point $o \in X$. Recent results of Ranjan and Shah~\cite{rs2003} imply that for any $p \in X$ the map $\Phi_p : S_p X \rightarrow \partial \hat{X}$ given by
\begin{align*}
\Phi_p(v)(x) = b_v(x)-b_v(o)
\end{align*}
is a homeomorphsim when $X$ is simply connected, non-compact, and harmonic. Our first main result is the following.

\begin{theorem}
\label{thm:main}
Suppose $X$ is a non-flat, non-compact, simply connected harmonic manifold. Then the map $id:X \rightarrow X$ extends to a homeomorphism of the Busemann and Martin compactifications. On the boundaries this map is given by
\begin{align*}
\xi \in \partial \hat{X} \rightarrow e^{-h_{vol} \xi} \in \partial_{\Delta} X.
\end{align*}
Under this identification of $\partial \hat{X}$ and $\partial_{\Delta}X$, the harmonic measures are given by $\mu_p = (\Phi_p)_* \lambda_p$ where $\lambda_p$ is the Lebesque measure on $S_p X$. 
\end{theorem}

Our second result concerns finite volume harmonic manifolds without conjugate points.

\begin{theorem}
\label{thm:compact}
Suppose $M$ is a non-flat finite volume harmonic manifold without conjugate points. Let $X$ be the universal Riemannian cover of $M$ with deck transformations $\Gamma=\pi_1(M) \subset \text{Isom}(X)$. Then
\begin{enumerate}
\item up to scaling, there is a unique $\Gamma$-Patterson-Sullivan measure $\{ \nu_x : x \in X\}$ on the Busemann boundary of $X$,
\item under the natural identification of the Busemann boundary and Martin boundary the Patterson-Sullivan measures coincides with the harmonic measures.
\end{enumerate}
\end{theorem}

Using the above results and an extension of a result of Eberlein, we will prove that the geodesic flow is topologically transitive.

\begin{corollary}
\label{cor:trans}
If $M$ is a non-flat finite volume harmonic manifold without conjugate points, then the geodesic flow is topologically transitive on $SM$.
\end{corollary}

\subsection{History:} In 1944 Lichnerowicz~\cite{lichnerowicz44} proved that every harmonic manifold of dimension three or less is either flat or a rank one locally symmetric space. Lichnerowicz conjectured that the same is true for four dimensional harmonic manifolds. In 1949 Walker~\cite{W1949} verified this conjecture. Since then it has been an open problem to characterize the harmonic manifolds in higher dimensions. In 1990, Szab{\'o}~\cite{szabo90} proved that any compact simply connected harmonic manifold is a rank one symmetric space of compact type.

Based on the work of Lichnerowicz, Szab{\'o}, and Walker it was natural to suspect that all harmonic manifolds are either flat or a rank one locally symmetric space. However in 1992, Damek and Ricci~\cite{DR92} constructed nonsymmetric homogeneous harmonic manifolds. In 2006, Heber~\cite{heber06} showed that any homogeneous harmonic manifold is either of the type constructed by Damek and Ricci, flat, or a rank one symmetric space. 

The nonsymmetric examples of Damek and Ricci do not have compact quotients and it is reasonable to suspect that compact harmonic manifolds are locally symmetric. This has been verified in a number of cases. In 1995 Besson, Courtois, and Gallot~\cite{BCG95,BCG96} applied their minimal entropy rigidity results to show that any simply connected negatively curved harmonic manifold with compact quotient is a rank one symmetric space. This corollary of their work used deep results of Foulon and Labourie~\cite{FL92} and Benoist, Foulon, and Labourie~\cite{BFL92}.

Recent results of Nikolayevsky~\cite{nik2005} and Ranjan and Shah~\cite{rs2002} imply that any harmonic manifold without conjugate points and zero volume growth entropy is flat. In 2009, Knieper~\cite{knieper11} proved that any compact non-flat harmonic manifold without conjugate points is a rank one locally symmetric space assuming either nonpositive curvature (or more generally no focal points) or Gromov hyperbolic fundamental group. 

Knieper also extended the well known definition of the rank of a nonpositively curved manifold to manifolds without conjugate points. He then proved that every rank one compact harmonic manifold is locally symmetric. It is well known that the Martin and Busemann boundary differ for nonpositively curved manifolds of higher rank. Moreover the geodesic flow is not topologically transitive for these manifolds. In particular, the results of this paper show that higher rank harmonic manifolds without conjugate points (if they exist at all) are interesting Riemannian manifolds.

\subsection*{Acknowledgements}

I would like to thank Ralf Spatzier for many helpful conversations.  I also thank the National Science Foundation for support through the grant DMS-0602191.

\section{Preliminaries}

In this article all Riemannian manifolds are assumed to be complete and have $C^\infty$ metrics. Given a Riemannian manifold $M$, the unit tangent bundle will be denoted by $SM$. Given a vector $v \in SM$, $\gamma_v:\RR \rightarrow M$ will denote the unique geodesic with $\gamma_v^\prime(0)=v$. Finally $g^t:SM \rightarrow SM$ will denote the geodesic flow on $SM$. 

We begin our preliminaries by discussing two compactifications of non-compact Riemannian manifolds. Wang's survey paper~\cite{wang2011} provides additional details on these compactifications and Ancona's survery paper~\cite{ancona1990} is an excellent reference on the Martin Boundary. 

\subsection{The Busemann compactification and Patterson-Sullivan measures}

Suppose $X$ is a non-compact Riemannian manifold and fix $o \in X$. Let $C(X)$ be the space of continuous functions  $X \rightarrow \RR$ with the topology of uniform convergence on compact sets. Consider the embedding $X \rightarrow C(X)$ by mapping $y \in X$ to the function $b_y$ given by
\begin{align*}
b_y(x) = d(x,y)-d(y,o).
\end{align*}
As each $b_y$ is 1-Lipschitz, the image of $X$ under this map is a relatively compact subset of $C(X)$. We then define the \emph{Busemann compactification} $\hat{X}$ of $X$ to be the compactification of $X$ in $C(X)$ with respect to this embedding. The Busemann boundary of $X$ is the set $\partial \hat{X} = \hat{X} \setminus X$. Notice that the action of $\text{Isom}(X)$ extends to an action by homeomorphisms on $\hat{X}$. For $\xi \in \partial \hat{X}$ and $g \in \text{Isom}(X)$ this action is given by:
\begin{align*}
(g \cdot \xi)(x) = \xi(g^{-1} x) -\xi(g^{-1}o).
\end{align*}

Ledrappier and Wang~\cite{LW10} considered Patterson-Sullivan measures on the Busemann boundary of general non-compact manifolds. Let $M$ be a Riemannian manifold with non-compact universal Riemannian cover $X$ and let $\Gamma=\pi_1(M) \subset \text{Isom}(X)$ be the deck transformations of the covering $X \rightarrow M$.  For $p \in X$ consider the following limit:
\begin{align*}
\lim_{r\rightarrow \infty} \frac{\log \text{Vol}_X  B_r(p)}{r}.
\end{align*}
When this limit exists and does not depend on the choice of $p \in X$, we denote this number by $h_{vol}$ and call it the \emph{volume growth entropy of $X$}. When $M$ is compact, Manning~\cite{manning79} showed that the limit above always exists and is independent of $p$. Further Manning showed that when $M$ is compact and nonpositively curved $h_{top}=h_{vol}$, where $h_{top}$ is the topological entropy of the geodesic flow on $SM$. Freire and Ma\~{n}\'{e}~\cite{FM82} generalized this last result and showed that $h_{top}=h_{vol}$ when $M$ is compact and has no conjugate points. 

\begin{definition}
Suppose $M$ is a Riemannian manifold. Let $X$ be the universal Riemannian cover of $M$ with deck transformations $\Gamma=\pi_1(M) \subset \text{Isom}(X)$. If $h_{vol}$ exists and is finite, a family of measures $\{ \nu_x : x \in X\}$ on $\partial \hat{X}$ is a (normalized) \emph{$\Gamma$-Patterson-Sullivan measure} if
\begin{enumerate}
\item $\nu_o(\partial \hat{X})=1$,
\item for any $x,y \in X$ the measures $\nu_x,\nu_y$ have the same measure class and satisify 
\begin{align*}
\frac{d\nu_x}{d\nu_y}(\xi) = e^{-h_{vol}(\xi(x)-\xi(y))},
\end{align*}
\item for any $g \in \Gamma$, $\nu_{gx} = g_*\nu_x$.
\end{enumerate}
\end{definition}

\subsection{The Martin compactification}

Suppose $(X,g)$ is a complete non-compact Riemannian manifold and $\Delta: L^2(X) \rightarrow L^2(X)$ is the Laplace-Beltrami operator on $X$. This is a symmetric operator and the bottom of the spectrum is given by
\begin{align*}
\lambda_{min}=\lambda_{min}(X,g) = \inf \left\{ \frac{\int \norm{\nabla f}^2 dx}{\int \abs{f}^2 dx} : f \in C^{\infty}_K(X)\right\}.
\end{align*}
If $\lambda_{min} >0$ or more generally $X$ is nonparabolic then $X$ has a \emph{minimal positive Green's function} $G(x,y)$ this is a positive function on $M \times M \setminus \{ (x,x) \in M\times M\}$ such that $\Delta_x G(x,y) = -\delta_y$ in the sense of distributions and $G(x,y)$ is the minimal function with these properties. 

Using this Green's function we can define the Martin compactification as follows. Fix $o \in X$ and consider the space of positive harmonic functions on $X$ normalized at $o$:
\begin{align*}
\mathcal{K}_o=\mathcal{K}_o(X) = \{ h : \ h>0, \ \Delta h =0, \ h(o)=1\} \subset C(X).
\end{align*}
By Harnack's inequality this is a compact, convex set in the topology of uniform convergence on compact sets.

Next consider the map
\begin{align*}
y \in X \rightarrow k_y(x) = \frac{G(x,y)}{G(o,y)}.
\end{align*}
Notice that $k_y$ is a harmonic function on $X \setminus \{y\}$. A sequence $y_n \in X$ is said to be a \emph{Martin sequence} if $k_{y_n}$ converges locally uniformly to a function $h$. Using elliptic theory, every sequence $k_{y_n}$ with $y_n \rightarrow \infty$ has a subsequence that converges and any limit point is in $\mathcal{K}_o$. The \emph{Martin boundary} $\partial_{\Delta} X \subset \mathcal{K}_o$ is defined to be the end points of all Martin sequences. This set will be compact in $\mathcal{K}_o$. The \emph{Martin compactification} of $X$ is the set $X \sqcup \partial_{\Delta} X$ with the unique topology making it a compactification of $X$. Notice that the action of $\text{Isom}(X)$ extends to an action by homeomorphisms on $X \cup \partial_{\Delta} X$. For $h \in \partial_{\Delta} X$ and $g \in \text{Isom}(X)$ this action is given by:
\begin{align*}
(g \cdot h)(x) = h(g^{-1} x)/ h(g^{-1}o).
\end{align*}
See Ancona's survey~\cite[Section II.2]{ancona1990} for proofs of the assertions in the above paragraph.

As $\mathcal{K}_o$ is a compact, convex set it has distinguished points: the extreme points. In this context, these points are often called minimal.

\begin{definition}
The \emph{minimal points} of $\mathcal{K}_o$ are the functions $h \in \mathcal{K}_o$ such that if $h_1,h_2 \in \mathcal{K}_o$, $\lambda \in [0,1]$, and $h=\lambda h_1+(1-\lambda)h_2$ then $h=h_1=h_2$.
\end{definition}

Let $\mathcal{K}_o^* \subset \mathcal{K}_o$ denote the minimal points of $\mathcal{K}_o$. It is well known that the Martin boundary contains all the minimal points (see for instance~\cite[Theorem 2.1]{ancona1990}), i.e.
\begin{align*}
\mathcal{K}_o^* \subset \partial_{\Delta} X \subset \mathcal{K}_o,
\end{align*}
and in particular any Borel measure on $\mathcal{K}_o^*$ can be realized as a Borel measure on $\partial_{\Delta} X$. Using Choquet theory (see for instance~\cite[Section II.1]{ancona1990}), for any $f \in \mathcal{K}_o$ there exists a unique Borel measure $\mu$ on $\mathcal{K}_o^*$ such that 
\begin{align*}
f = \int_{\mathcal{K}_o^*} h d\mu(h).
\end{align*}
This gives rise to the harmonic measure on $\partial_{\Delta} X$. Let $\mu$ be the unique Borel measure on $\partial_{\Delta} X$ supported on $\mathcal{K}_o^*$ and such that
\begin{align*}
1 = \int_{\partial_{\Delta} X} h(x) d\mu(h)
\end{align*}
for all $x \in X$. 

\begin{definition}
With the notation above, the \emph{harmonic measure} on $\partial_{\Delta} X$ is the family $\{ \mu_x : x \in X\}$ where $d\mu_x (h) = h(x) d\mu(h)$.
\end{definition}

Next define $\Pi = \supp(\mu)$ to be the \emph{Poisson boundary of $X$}. A motivation for the definitions above is contained in the following proposition.

\begin{proposition}
With the notation above, for each $f \in L^{\infty}(\Pi,\mu)$ the function
\begin{align*}
 H_f(x) = \int_{\Pi} f(h) h(x) d\mu(h)=\mu_x(f)
\end{align*}
is an bounded harmonic function on $X$.
\end{proposition}

In some cases the Martin boundary and Poisson boundary are familiar compactifications. 

\begin{theorem}
\cite[Section V, Theorem 6.2]{ancona1990}
Suppose that $X$ is Gromov hyperbolic and $\lambda_{min} >0$. Then the identity map $id:X \rightarrow X$ extends to a homeomorphism of the Martin boundary and the geometric boundary. Further, the Poisson boundary coincides with the Martin boundary.
\end{theorem}

Anderson and Schoen~\cite{AS85} proved the above theorem in the case of simply connected manifolds with pinched negative curvature.

\subsection{Harmonic Manifolds}

In this subsection we record some properties of harmonic manifolds that will be useful later. First we recall the definition of a harmonic manifold. 

\begin{definition}
\label{defn:harm}
A complete Riemannian manifold $M$ is called \emph{harmonic} if for any $y \in M$ there exists a neighborhood $U$ of $y$ such that the function 
\begin{align*}
x \in U\setminus \{y\} \rightarrow \Delta_x d(x,y) \in \RR
\end{align*}
depends only on $d(x,y)$.
\end{definition}

If $M$ is a Riemannian manifold, $f:M \rightarrow \RR$ is smooth, and $\norm{\nabla f} \equiv 1 $ near $p$, then the mean curvature of the hypersurface $f^{-1}(0)$ at $p \in f^{-1}(0)$ is given by $\Delta f(p)$. So the definition above says that small geodesic spheres have constant mean curvature.

The proof of next lemma can be found in the introduction of Szabo's paper~\cite{szabo90}.

\begin{lemma}
Let $M$ be a harmonic manifold. Then the following are equivalent:
\begin{enumerate}
\item $M$ has no conjugate points,
\item the universal cover of $M$ is non-compact,
\end{enumerate}
In any of above hold and $X$ is the universal Riemannian cover of $M$ then 
\begin{align*}
(x,y) \in X \times X \setminus \{(x,x) : x \in X \} \rightarrow \Delta_x d(x,y) \in \RR
\end{align*}
is a function of $d(x,y)$. 
\end{lemma}

Hence all geodesics spheres in $X$ of a fixed radius $r>0$ have the same constant mean curvature. Using this observation one can prove the next lemma (see Knieper~\cite[Section 2]{knieper11} for details).

\begin{lemma}
\label{lem:volume_ratio}
Let $X$ be a simply connected harmonic manifold without conjugate points. Then the volume growth entropy, $h_{vol}$, exists and is finite. Further, if $S_r(p)$ is the geodesic sphere of radius $r$ about $p \in X$, then $\text{Vol}_X S_r(p)$ does not depend on $p \in X$ and if $\Theta(r) = \text{Vol}_X S_r(p)$ then
\begin{align*}
\lim_{r \rightarrow \infty} \frac{\Theta^\prime(r)}{\Theta(r)} = h_{vol}.
\end{align*}
\end{lemma}

The case in which $h_{vol}=0$ is well understood due to recent results of Nikolayevsky~\cite[Theorem 2]{nik2005} and Ranjan and Shah~\cite[Theorem 4.2]{rs2002}.

\begin{theorem}\cite{nik2005,rs2002}
\label{thm:flat}
Let $X$ be a simply connected harmonic manifold without conjugate points. If $h_{vol}=0$ then $X$ is isometric to $\RR^n$.
\end{theorem}

If $X$ has no conjugate points and $v \in SX$ then there is a function $b_v$ associated to $v$:
\begin{align*}
b_v(x) = \lim_{t \rightarrow \infty} d(x,\gamma_v(t))-t.
\end{align*}
Ranjan and Shah~\cite[Theorem 3.1, Theorem 5.1]{rs2003} proved the following useful results about these functions.

\begin{lemma}\cite{rs2003}
\label{lem:harm_basic}
Let $X$ be a simply connected harmonic manifold without conjugate points, then with the notation above
\begin{enumerate}
\item $\Delta b_v \equiv h_{vol}$ for each $v \in SX$, 
\item $e^{-h_{vol} b_v}$ is harmonic for each $v \in SX$,
\item the map $(v,x) \rightarrow \nabla^n b_v(x)$ is continuous for any $n \geq 0$,
\end{enumerate}
\end{lemma}

Using properties of harmonic functions, Ranjan and Shah~\cite[Corollary 5.1]{rs2003} proved that $\partial \hat{X} =\{ b_v-b_v(o) : v \in S_pX\}$ for any $p \in X$ when $X$ is harmonic, simply connected, and without conjugate points. Using the continuity of $v \rightarrow b_v$ we provide a different proof. For $v \in SX$ and $t>0$ let $b_v^t \in C(X)$ be defined by
\begin{align*}
b_v^t(x) = d(x,\gamma_v(t))-t.
\end{align*}

\begin{proposition}
\label{prop:phi_map}
Suppose $X$ is a simply connected manifold without conjugate points. If the map $\Phi:SX \rightarrow C(X)$ given by $\Phi(v)=b_v$ is continuous, then the convergence 
\begin{align*}
\lim_{t \rightarrow \infty} b_v^t(x) = b_v(x)
\end{align*}
is locally uniform in $x \in X$ and $v \in SX$. Moreover, for any $p \in X$ the map $\Phi_p:S_pX \rightarrow \partial \hat{X}$ given by $\Phi_p(v) = b_v-b_v(o)$ is a homeomorphism. 
\end{proposition}

\begin{remark} Because of our choice of normalization unless $v\in S_oX$ the function $b_v$ will not be in $\partial \hat{X}$. \end{remark}

\begin{proof}
The triangle inequality shows that $b_v^{T_1}(x) \geq b_v^{T_2}(x)$ for $T_1 < T_2$. As the map $(x,v) \rightarrow b_v(x)$ is continuous, Dini's Theorem then implies that the convergence $b_v^t(x) \rightarrow b_v(x)$ is locally uniform in $v \in SX$ and $x \in X$. 

Now suppose $b_{y_n}$ converges to a Busemann function $\xi \in \partial \hat{X}$, then using the completeness of $X$ there exists $v_n \in S_pX$ and $t_n \in [0,\infty)$ such that $b_{y_n} = b_{v_n}^{t_n}-b_{v_n}^{t_n}(o)$. By passing to a subsequence we can assume $v_n \rightarrow v$ and then using the fact that the convergence $b_w^t(x) \rightarrow b_w(x)$ is uniform in $w \in S_pX$ and locally uniform in $x \in X$  we see that $\xi = b_v-b_v(o)$. This show that $\Phi_p$ is onto. By~\cite[Proposition 1]{eschen77} each $b_v$ is $C^1$ and $\nabla b_v(p)=-v$ so $\Phi_p$ is one-to-one. Finally as $S_pX$ is compact, $\Phi_p:S_pX \rightarrow \partial \hat{X}$ is a homeomorphism.  
\end{proof} 

Motivation for using the word ``harmonic'' in Definition~\ref{defn:harm} can be found in the next lemma.

\begin{lemma}\cite[Theorem 3]{W50}\label{lem:MVT}
Let $X$ be a simply connected, non-compact harmonic manifold. Suppose $h$ is a harmonic function on a open subset $\OO \subset X$. Then for $\epsilon>0$ such that $S_\epsilon(p) \subset \OO$ we have
\begin{align*}
h(p) = \frac{1}{\text{Vol}_X  S_\epsilon(p)}\int_{S_\epsilon(p)} h(y)dy
\end{align*}
where $dy$ is the induced measure on the geodesic sphere $S_\epsilon(p)$.
\end{lemma}

\begin{remark}
Lemma~\ref{lem:MVT} holds for general harmonic manifolds when $\epsilon$ is smaller than the injectivity radius at $p$.
\end{remark}

\begin{lemma}
\label{lem:green}
Let $X$ be a simply connected harmonic manifold without conjugate points. Then there is a minimal Green's function $G$ and $G(x,y)$ is a function of only the distance $d(x,y)$.
\end{lemma}

\begin{proof}
Let $\lambda_{min}$ be the bottom of the spectrum of the Laplace-Beltrami operator on $(X,g)$. If $\lambda_{min} >0$, then the minimal Green's function exists and is given by
\begin{align*}
G(x,y) = \int_{0}^{\infty} p_t(x,y) dt
\end{align*}
where $p_t(x,y)$ is the heat kernel (see for instance~\cite[Theorem 13.4]{grig09}). Further Szab{\'o}~\cite[Theorem 5.1]{szabo90} has proven that $p_t(x,y)$ is radial when $X$ is harmonic. 

By Theorem~\ref{thm:flat}, it is enough to consider the case in which $h_{vol}>0$. We will now show that $4\lambda_{min}=h_{vol}^2$ (this fact is well known) which will complete the proof.

By approximating $e^{-sd(o,x)}$ by a compactly supported $C^2$ functions when $s>h_{vol}/2$ we see that $4\lambda_{min} \leq h_{vol}^2$. For the other direction we use an argument of Grigor'yan~\cite[Theorem 11.17]{grig09}: let $\xi$ be a Busemann function then $\norm{\nabla \xi} \equiv1$ and $\Delta\xi \equiv h_{vol}$. If $f \in C^{\infty}_K(X)$, using integration by parts and Cauchy's inequality we have
\begin{align*}
h_{vol} \int_X \abs{f}^2 dx = \int_X \Delta \xi \abs{f}^2 = -2\int_X f \left\langle \nabla \xi, \nabla f\right\rangle dx \leq 2\int_X \abs{f} \norm{\nabla f} dx.
\end{align*}
Finally H\"{o}lder's inequality implies that 
\begin{align*}
h_{vol}^2 \leq 4 \frac{ \int_X \norm{\nabla f}^2 dx }{\int_X \abs{f}^2 dx}.
\end{align*}
As $f \in C_0^{\infty}(X)$ was arbitrary we have $h_{vol}^2 \leq 4\lambda_{min}$. 
\end{proof}

\section{Proof of Theorem~\ref{thm:main} and Theorem~\ref{thm:compact}}

\subsection{The Martin boundary and Busemann boundary coincide}

In this subsection we prove the following.

\begin{theorem}
\label{thm:bds}
Suppose $X$ is a non-flat, non-compact, simply connected harmonic manifold. Then the map $id:X \rightarrow X$ extends to a homeomorphism of the Busemann and Martin compactifications. On the boundaries this map is given by
\begin{align*}
\xi \in \partial \hat{X} \rightarrow e^{-h_{vol} \xi} \in \partial_{\Delta} X.
\end{align*}
\end{theorem}

Let $X$ be as in Theorem~\ref{thm:bds}. Then by Lemma~\ref{lem:green} the Green's function $G(x,y)$ depends only on $d(x,y)$. Let $\wt{G}:\RR_{>0} \rightarrow \RR_{>0}$ be the function such that
\begin{align*}
G(x,y)=\wt{G}(d(x,y)).
\end{align*}

\begin{lemma}
With the notation above,
\begin{align*}
\lim_{r\rightarrow \infty} \wt{G}^\prime(r)=0.
\end{align*}
\end{lemma}

\begin{proof}
Suppose for a contradiction that there exists $r_n \rightarrow \infty$ and $\epsilon>0$ such that $\abs{\wt{G}^\prime(r_n)}>\epsilon$ for all $n$. Pick $x_n \in X$ such that $r_n = d(o,x_n)$ and consider the sequence $k_n:=k_{x_n}$. By passing to a subsequence we may suppose $k_{n} \rightarrow h \in \partial_{\Delta} X$ locally uniformly. By Lemma~\ref{lem:harm_conv}, this implies that $\nabla k_n \rightarrow \nabla h$ locally uniformly. But
\begin{align*}
\nabla k_n(x) = \frac{\nabla_x G(x,x_n)}{G(o,x_n)} \text{ and so } \nabla k_n(o) = \frac{\wt{G}^\prime(r_n)}{\wt{G}(r_n)} (\nabla_x d(x,x_n))|_{x=o}.
\end{align*}
Now $\norm{\nabla k_n(o)} = \abs{\frac{\wt{G}^\prime(r_n)}{\wt{G}(r_n)}}$ converges to $\norm{\nabla h(o)}$ and $\wt{G}(r_n)$ converges to zero (see for instance the proof of Theorem 13.4 in~\cite{grig09}), so we must have that $\wt{G}^\prime(r_n)$ converges to zero. Which is a contradiction.
\end{proof}

\begin{lemma}
\label{lem:green_fcn_exp_decay}
With the notation above,
\begin{align*}
\lim_{r\rightarrow \infty} \frac{\wt{G}^\prime(r)}{\wt{G}(r)} = -h_{vol}
\end{align*}
\end{lemma}

\begin{proof}
As $G(x,y)=\wt{G}(d(x,y))$ is a harmonic function in $x$ on $X\setminus\{y\}$ and depends only on $d(x,y)$, equation (1.3) in~\cite{szabo90} implies that
\begin{align*}
0 = -\Delta_x G(x,y) = \wt{G}^{\prime\prime}(r) + \frac{ \Theta^\prime(r)}{\Theta(r)}\wt{G}^\prime(r) \text{ for $r=d(x,y)>0$}.
\end{align*}
For $0 < r < R$ we then have:
\begin{align*}
-\sup\{ \Theta^\prime(s)/\Theta(s) : s \in [r,R]\} < \frac{ \int_{r}^R \wt{G}^{\prime\prime}(s)ds}{\int_{r}^R \wt{G}^{\prime}(s)ds} < -\inf\{ \Theta^\prime(s)/\Theta(s) : s \in [r,R]\}
\end{align*}
which becomes 
\begin{align*}
-\sup\{ \Theta^\prime(s)/\Theta(s) : s \in [r,R]\} < \frac{ \wt{G}^{\prime}(R)-\wt{G}^{\prime}(r)}{\wt{G}(R)-\wt{G}(r)} < -\inf\{ \Theta^\prime(s)/\Theta(s) : s \in [r,R]\}
\end{align*}
sending $R \rightarrow \infty$ yields
\begin{align*}
-\sup\{\Theta^\prime(s)/\Theta(s) : s \in [r,-\infty)\} < \frac{ \wt{G}^{\prime}(r)}{\wt{G}(r)} < -\inf\{ \Theta^\prime(s)/\Theta(s) : s \in [r,\infty)\}.
\end{align*}
The lemma now follows from the fact that $\Theta^\prime(s)/\Theta(s) \rightarrow h_{vol}$ as $s \rightarrow \infty$.
\end{proof}

\begin{proof}[Proof of Theorem~\ref{thm:bds}]
It is enough to show that the sequence
\begin{align*}
b_{x_n}(x) = d(x,x_n)-d(x_n,x)
\end{align*}
converges locally uniformly to $\xi$ if and only if 
\begin{align*}
k_{x_n}(x) = \frac{G(x,x_n)}{G(o,x_n)}
\end{align*}
converges locally uniformly to $e^{-h_{vol}\xi}$. 

By Lemma~\ref{lem:harm_conv}, $k_{x_n}$ converges locally uniformly if and only if $\nabla k_{x_n}$ converges locally uniformly. Moreover
\begin{align}
\label{eq:1234}
\nabla k_{x_n}(x) = \frac{\wt{G}^\prime(d(x,x_n))}{\wt{G}(d(x,x_n))} k_{x_n}(x) \nabla d(x,x_n).
\end{align}
Lemma~\ref{lem:green_fcn_exp_decay} implies that
\begin{align*}
\frac{\wt{G}^\prime(d(x,x_n))}{\wt{G}(d(x,x_n))}
\end{align*}
converges locally uniformly to $h_{vol}$ if $x_n \rightarrow \infty$. So $k_{x_n}(x)$ converges locally uniformly if and only if $\nabla d(x,x_n)=\nabla b_{x_n}$ converges locally uniformly. But by~\cite[Chapter 6]{besse78}, $X$ has bounded sectional curvature and so by Lemma~\ref{lem:conv} this latter condition is equivalent to $b_{x_n}$ converging locally uniformly.
\end{proof}

\subsection{Each element of the Martin boundary is minimal}

In this subsection we will prove the following.

\begin{proposition}
\label{prop:minimal}
Suppose $X$ is a non-flat, non-compact, simply connected harmonic manifold. Then each $h \in \partial_\Delta X$ is minimal.
\end{proposition}

\begin{lemma}
\label{lem:buse_growth}
With $X$ as in Proposition~\ref{prop:minimal}, for all $v,w \in S_oX$ and $t \geq 0$
\begin{align*}
b_v(\gamma_w(t)) \geq -t
\end{align*}
with equality if and only if $v= w$. 
\end{lemma}

\begin{proof}
Notice that $b_v$ and $b_w$ are smooth and $\nabla b_v(o) = -v$, $\nabla b_w(o) = -w$, $\norm{\nabla b_v } \equiv 1$, and  $\norm{\nabla b_w} \equiv 1$. Then
\begin{align*}
b_v(\gamma_w(t)) 
& = \int_0^t \frac{d}{ds} b_v(\gamma_w(s)) ds = \int_0^t g\big(\nabla b_v(\gamma_w(s)), \dot{\gamma}_w(s)\big) ds\\
& = \int_0^t -g\big(\nabla b_v(\gamma_w(s)), \nabla b_w(\gamma_w(s))\big) ds \geq \int_0^t -1 ds = -t.
\end{align*}
with equality if and only if $v=w$.
\end{proof}

\begin{proof}[Proof of Proposition~\ref{prop:minimal}]
Using Theorem~\ref{thm:bds} and Proposition~\ref{prop:phi_map}, we have a homeomorphism $S_oX \rightarrow \partial_{\Delta} X$ given by $v \rightarrow e^{-h_{vol} b_v}$. Using this identification it is enough to prove the following: if $w \in S_oX$ and there exists a probability measure $\nu$ on $S_oX$ such that
\begin{align*}
e^{-h_{vol} b_w} = \int_{S_oX} e^{-h_{vol} b_v} d\nu(v)
\end{align*}
then $\nu = \delta_w$. To see this, evaluate the above expression at $\gamma_w(t)$ and use Lemma~\ref{lem:buse_growth} to obtain
\begin{align*}
e^{h_{vol}t} = \int_{S_oX} e^{-h_{vol} b_v(\gamma_w(t))} d \nu(v)  \leq \int_{S_oX} e^{h_{vol} t} d \nu(v)= e^{h_{vol}t}.
\end{align*}
Thus we must have $\nu = \delta_w$ and hence $e^{-h_{vol} b_w}$ is a minimal element of the Martin boundary.
\end{proof}

\subsection{The harmonic measures have full support}

For a complete simply connected Riemannian manifold $(X,g)$ without conjugate points and a point $p \in X$ we have normal coordinates on $X$ coming from the diffeomorphism
\begin{align*}
\exp_p : T_p X \rightarrow X.
\end{align*}
Now consider the group $O(g_p) \subset GL(T_pM)$ of linear maps preserving the inner product $g_p$ on $T_pM$. Then $O(g_p)$ acts on $X$ by
\begin{align}
\label{eq:action_one}
T \cdot \exp_p(v)=\exp_p(Tv)
\end{align}
Notice that if in addition $X$ is a harmonic manifold then the Riemannian volume form $dx$ is invariant under this $O(g_p)$ action.

\begin{lemma}
\label{lem:polar_coor_ext}
Suppose $X$ is a complete simply connected Riemannian manifold without conjugate points. If the map $\Phi:SX \rightarrow C(X)$ given by $\Phi(v) = b_v$ is continuous, then the action of $O(g_p)$ on $X$ extends to an action of $O(g_p)$ on $\hat{X}$ by homeomorphisms such that
\begin{align}
\label{eq:action_two}
T \cdot (b_v-b_v(o)) = b_{Tv}-b_{Tv}(o)
\end{align}
for $T \in O(g_p)$ and $v \in S_pX$.
\end{lemma}

Recall that Proposition~\ref{prop:phi_map} implies that $\partial \hat{X} = \{ b_v -b_v(o): v \in S_p X\}$ and that the convergence $b_v^t(x) \rightarrow b_v(x)$ is locally uniform in $x \in X$ and $v \in SX$.

\begin{proof}
First define an action of $O(g_p)$ on $\hat{X}$ using equations~\ref{eq:action_one} and~\ref{eq:action_two}. Now fix $T \in O(g_p)$. The fact that the convergence $b_v^t(x) \rightarrow b_v(x)$ is uniform in $v \in S_pX$ and locally uniform in $x \in X$ implies that the map $f_T:\hat{X} \rightarrow \hat{X}$ given by $f_T(\xi) = T \cdot \xi$ is continuous.
\end{proof}

\begin{theorem}
\label{thm:full_support}
Let $X$ be a simply connected non-flat harmonic manifold without conjugate points. If $\{ \mu_x : x\in X\}$ is the harmonic measure on $\partial \hat{X}$ then for each $p \in X$ the measure $\mu_p$ is invariant under the $O(g_p)$ action described above. In particular, $\mu_p$ has full support and $\mu_p = (\Phi_p)_*\lambda_p$ where $\lambda_p$ is the Lebesque measure on $S_pX$.
\end{theorem}

\begin{proof}
For each $x \in X$ and $s>h_{vol}$ consider the probability measure $\mu_x^s$ on $\hat{X}$ given by
\begin{align*}
\mu_x^s(f)= \frac{\int_X f(b_y)e^{-sd(x,y)} dy}{\int_X e^{-sd(x,y)} dy}
\end{align*}
for $f \in C(\hat{X})$. Since $X$ is a harmonic manifold, the volume form $dx$ is invariant under the $O(g_p)$ action described above and so $\mu_p^s$ is also invariant under the $O(g_p)$ action on $\hat{X}$. 

By weak-$*$ compactness there exists a sequence $s_n \searrow h_{vol}$ and a probability measure $\mu_o$ on $\hat{X}$ such that $\mu_o^{s_n} \rightarrow \mu_o$ weakly. By a result of Nikolayevsky~\cite[Theorem 2]{nik2005}
\begin{align*}
\lim_{r \rightarrow \infty} \frac{\text{Vol}_X S_r(o)}{e^{h_{vol}r}} \in [0, \infty]
\end{align*}
exists and is nonzero. Hence 
\begin{align*}
\lim_{n \rightarrow \infty} \int_X e^{-s_n d(x,y)} dy = + \infty
\end{align*}
and so $\mu_o$ is supported on $\partial \hat{X}$.

Let $B:\hat{X} \times X \rightarrow \RR$ be the continuous extension of $B(y,x)=d(x,y)-d(y,o)$ to $\hat{X} \times X$. Using the fact that $\text{Vol}_X S_r(x)=\text{Vol}_X S_r(o)$ for all $x \in X$, we have that 
\begin{align*}
d\mu_x^s(\xi) = e^{-s B(\xi,x)}d\mu_o^s(\xi).
\end{align*}
For $x$ fixed, $e^{-sB(\xi,x)}$ converges uniformly in $\xi \in \hat{X}$ to $e^{-h_{vol}B(\xi,x)}$ as $s \rightarrow h_{vol}$. So $\mu_x^{s_n}$ converges weakly to a measure $\mu_x$ on $\partial \hat{X}$.

Now by construction for all $x \in X$ we have $d\mu_x(\xi) =e^{-h_{vol}\xi(x)}\mu_o(\xi)$ and $\mu_x(\partial \hat{X}) = 1$ so
\begin{align*}
1 = \int_{\partial \hat{X}} d\mu_x(\xi) = \int_{\partial \hat{X}}e^{-h_{vol}\xi(x)}d\mu_o(\xi)
\end{align*}
and $\{\mu_x : x \in X\}$ is the unique harmonic measure on $\partial \hat{X}$. Further, by construction, the measure $\mu_p$ is invariant under the $O(g_p)$ action described at the start of this section.
\end{proof}

\subsection{The measures coincide}

In this section we prove the following proposition.

\begin{proposition}
Suppose $M$ is a finite volume non-flat harmonic manifold without conjugate points. Let $X$ be the universal Riemannian cover of $M$ with deck transformations $\Gamma=\pi_1(M) \subset \text{Isom}(X)$. Then
\begin{enumerate}
\item there is a unique $\Gamma$-Patterson-Sullivan measure $\{ \nu_x : x \in X\}$ on the Busemann boundary of $X$,
\item under the natural identification of the Busemann boundary and Martin boundary the Patterson-Sullivan measure coincides with the harmonic measure.
\end{enumerate}
\end{proposition}

\begin{proof}
Suppose $\{ \nu_x : x \in X\}$ is a $\Gamma$-Patterson-Sullivan measure then
\begin{align*}
d\nu_x(\xi) = e^{-h_{vol} \xi(x)}d\nu_o(\xi)
\end{align*}
and
\begin{align*}
\nu_x(\partial \hat{X}) = \int_{\partial \hat{X}} e^{-h_{vol} \xi(x)}d\nu_o(\xi)
\end{align*}
In particular the function $H(x)=\nu_x(\partial \hat{X})$, being a convex combination of harmonic functions, is a harmonic function on $X$. If $g \in \Gamma$ then $g_*\nu_x=\nu_{g x}$ and so the function $H$ is $\Gamma$-invariant and hence $H$ descends to a harmonic function on $M$. Since $M$ has finite volume, $H$ must be constant and so $H \equiv 1$.  Using the identification of $\partial \hat{X}$ and $\partial_{\Delta} X$, $\nu_o$ is then a measure on $\partial_{\Delta} X$ such that
\begin{align*}
1 = \int_{\partial_{\Delta} X} h(x) d\nu_o(h)
\end{align*}
for all $x \in X$. In particular, under the identification of $\partial \hat{X}$ and $\partial_{\Delta} X$, $\{\nu_x : x \in X\}$ is the harmonic measure on $\partial_{\Delta} X$.
\end{proof} 

\section{Extending a result of Eberlein}

The purpose of this section is to prove the following theorem.

\begin{theorem}
\label{thm:top_trans}
Suppose $M$ is a finite volume Riemannian manifold without conjugate points and with sectional curvature bounded from below. Let $X$ be the universal Riemannian cover of $M$ with deck transformations $\Gamma = \pi_1(M) \subset \text{Isom}(X)$. For $v \in SX$, let $U^s(v) \in \text{End}(v^{\bot})$ be the stable Riccati solution associated to $v$. If the map $v \rightarrow U^s(v)$ is continuous then the following are equivalent 
\begin{enumerate}
\item the geodesic flow is topologically transitive on $SM$,
\item there is a dense $\Gamma$-orbit in $\partial \hat{X}$.
\end{enumerate}
\end{theorem}

If $M$ has nonpositive curvature then the map $v \rightarrow U^s(v)$ is continuous and in this case the above theorem is due to Eberlein~\cite{eberlein73ncmII}. A proof of Eberlein's result can be found in Ballmann's book~\cite[Theorem 2.3]{Ball95}. We will closely follow this proof, but since working in the category of no conjugate points is complicated we provide all the details.

In Subsection 4.1 we introduce tensors along geodesics, in Subsection 4.2 we define the stable Riccati solutions, and in Subsection 4.3 we prove Theorem~\ref{thm:top_trans}. 

\subsection{Tensors along geodesics:} Given a Riemannian manifold $X$ and a geodesic $\gamma: I \rightarrow X$, let 
\begin{align*}
N_{\gamma} = \{ w \in T_{\gamma(t)} X: g(w,\gamma^\prime(t))=0\}
\end{align*}
be the normal bundle of $\gamma$. A (1,1)-tensor along $\gamma$ is a smooth bundle endomorphism of $N_{\gamma}$, i.e. a smooth map
\begin{align*}
t \in I \rightarrow Y(t) \in \text{End}(\gamma^\prime(t)^{\bot}).
\end{align*}
Given a smooth (1,1)-tensor $Y$ we can use the Levi-Civita connection to define the derivative of $Y$ as $Y^\prime = \nabla_{\gamma^\prime(t)} Y$. Then $Y^\prime$ is also a (1,1)-tensor.

Let $R$ be the curvature tensor on $X$. An example of a (1,1)-tensor is the Riemannian curvature tensor $t \rightarrow R(t)$ given by
\begin{align*}
R(t)x = R(x,\gamma^\prime(t))\gamma^\prime(t).
\end{align*}
  
An important class of (1,1)-tensors are the so called Jacobi tensors. A (1,1)-tensor $\JJ$ along a geodesic $\gamma:\RR \rightarrow X$ is called a \emph{Jacobi tensor} if 
\begin{align*}
\JJ^\prime(t) + R(t)\JJ(t) = 0.
\end{align*}
If $x_t$ is a parallel vector field along $\gamma$ orthogonal to $\gamma^\prime(t)$ then $\JJ(t)x_t$ will be a Jacobi field along $\gamma$.

\subsection{The Stable and Unstable Riccati solutions} 

In this subsection we introduce the stable Riccati solutions. Suppose $M$ is a complete Riemannian manifold without conjugate points, let $v \in SM$ and consider the Jacobi tensor $\JJ_{v,T}$ along $\gamma_v$ such that $\JJ_{v,T}(0)=Id$ and $\JJ_{v,T}(T)=0$. Let $U^s_T(v)=\JJ_{v,T}^\prime(0)$. We then have the following.

\begin{proposition}\cite{green58}
\label{prop:riccati_basic}
With the notation above,
\begin{enumerate}
\item $\JJ_{v,T}$ converges to a Jacobi tensor $\JJ_v^s$ along $\gamma_v$,
\item $U^s_T(v)$ converges monotonically to an endomorphism $U^s(v)$ in the sense that $U^s_{T_2}(v)-U^s_{T_1}(v)$ is positive definite for all $T_2>T_1>0$,
\item $U^s(v) = (\JJ_v^s)^\prime(0)$ and $U^s(g^tv) = (\JJ_v^s)^\prime(t)\JJ_v^s(t)^{-1}$,
\item the (1,1)-tensor $t \rightarrow U^s(g^tv)$ satisfies the Riccati equation:
\begin{align*}
(U^s)^\prime+(U^s)^2+R=0
\end{align*}
\end{enumerate}
where $R(t) = R(\cdot, \gamma_v^\prime(t))\gamma_v^\prime(t)$ is the curvature tensor along $\gamma_v$.  
\end{proposition}

The tensor $\JJ^s_v$ is called the \emph{stable Jacobi tensor along $\gamma_v$} and the map $v \rightarrow U^s(v)$ is called the \emph{stable Riccati solution}. 

Eberlein~\cite[Remark 2.10]{eberlein73I} appears to be the first to observe that the behavior of geodesics is related to the continuity of the map $v \rightarrow U^s(v)$. We end this subsection with two useful results. The first is due to Eschenburg and the second is due to Eschenburg and O'Sullivan.

\begin{lemma}
\label{lem:cont}
\cite{eschen77}
Let $X$ be a simply connected Riemannian manifold without conjugate points. If the map $v\rightarrow U^s(v)$ is continuous, then the map $v \in SX \rightarrow b_v \in C(X)$ is continuous.
\end{lemma}

\begin{proof}
As $U^s_T(v) \rightarrow U^s(v)$ monotonically, Dini's theorem and the hypothesis of the Lemma implies that the covergence $U^s_T(v) \rightarrow U^s(v)$ is locally uniform in $v \in SX$. Then by a result of Eschenburg~\cite[Theorem 1, (ii)]{eschen77} the map $(x,v) \in X \times SX \rightarrow \nabla b_v(x)$ is continuous. This implies that $v \rightarrow b_v \in C(X)$ is continuous.
\end{proof}

If $X$ is simply connected, has no conjugate points, and $p,q \in X$ are distinct let $\overline{pq}=\sigma^\prime(0) \in S_pX$ where $\sigma$ is the unit speed geodesic ray starting at $p$ and passing through $q$.

\begin{lemma}\cite{EO76}
\label{lem:ang}
Let $X$ be a simply connected Riemannian manifold without conjugate points and with sectional curvature bounded from below. Suppose that the map $v \rightarrow U^s(v)$ is continuous. If $x_n,y_n \rightarrow \infty$ in $X$ and $\sup_n d(x_n,y_n) < +\infty$ then for any $p \in X$, $\angle_p\left(\overline{px_n},\overline{py_n}\right)\rightarrow 0$.
\end{lemma}

Ranjan and Shah~\cite[Corollary 5.3]{rs2003} used results of Eschenburg and O'Sullivan~\cite{EO76} to prove the above lemma in the context of harmonic manifolds.

\begin{proof}
The lemma follows from Theorem 1, part (i) and Proposition 6 in~\cite{EO76}.
\end{proof}

\subsection{The proof of Theorem~\ref{thm:top_trans}}
 
Given a vector $v \in SX$ we define the endpoints of $v$ in $\partial \hat{X}$ to be
\begin{align*}
v(+\infty) & = \lim_{t \rightarrow +\infty} \gamma_v(t)=b_v-b_v(o) \\
v(-\infty) & = \lim_{t \rightarrow -\infty} \gamma_v(t)=b_{-v}-b_{-v}(o).
\end{align*}
These maps from $SX$ to $\partial \hat{X}$ are continuous and so if $v_n \rightarrow v$ then $v_n(\pm \infty) \rightarrow v(\pm \infty)$. 

The next three lemmas show that the Busemann boundary under the hypothesis of Theorem~\ref{thm:top_trans} behaves like the geodesic boundary of a complete $\text{CAT}(0)$ space. For $w \in SX$ and $t \in \RR$ let $b_w^t(x) = d(x,\gamma_w(t))-t$. We will repeatedly use the fact that 
\begin{align*}
\lim_{t \rightarrow \infty} b_w^t(x) = b_w(x)
\end{align*}
locally uniformly in $w \in SX$ and $x \in X$  under the hypothesis of Theorem~\ref{thm:top_trans}, this follows from Lemma~\ref{lem:cont} and Proposition~\ref{prop:phi_map}.

\begin{lemma}
\label{lem:convergence}
With the notation in Theorem~\ref{thm:top_trans}, suppose $v \in S_pX$ and $x_n=\exp_p(t_nv_n) \in X$ with $t_n >0$ and $v_n \in S_pX$. Then $x_n$ converges to $b_v-b_v(o) \in \partial \hat{X}$ if and only if $t_n \rightarrow \infty$ and $v_n \rightarrow v$.
\end{lemma}

\begin{proof}
Let $v,v_n \in S_pX$ and $t_n \in \RR$ be as in the statement of the lemma. 

First suppose that $v_n \rightarrow v$ and $t_n \rightarrow \infty$. Then the fact that $b_w^t(x) \rightarrow b_w(x)$ uniformly in $w \in S_pX$ and locally uniformly in $x \in X$ implies that 
\begin{align*}
b_{x_n} = b_{v_n}^{t_n}-b_{v_n}(o) \rightarrow b_v-b_v(o)
\end{align*}
locally uniformly on $X$. Hence $x_n \rightarrow b_v-b_v(o)$. The reverse direction is similar.
\end{proof}   

\begin{lemma}
\label{lem:bd_dist}
With the notation in Theorem~\ref{thm:top_trans}, let $x_n,y_n \in X$. If $x_n \rightarrow \xi \in \hat{X}$ and $\sup_n d(x_n,y_n)<+\infty$ then $ y_n \rightarrow \xi \in \partial \hat{X}$.
\end{lemma}

\begin{proof}
Suppose that $\xi = b_v-b_v(o)$ for some $v \in S_pX$ and $x_n = \exp_p(t_n v_n)$ with $t_n>0$ and $v_n\in S_pX$. Then by Lemma~\ref{lem:convergence}, $v_n \rightarrow v$. Now let $y_n = \exp_p(s_n u_n)$ with $s_n >0$ and $u_n \in S_pX$, then by Lemma~\ref{lem:ang} we have $\angle_p(u_n,v_n) \rightarrow 0$ and so $u_n \rightarrow v$. Finally Lemma~\ref{lem:convergence} implies that $y_n \rightarrow \xi$.
\end{proof}

\begin{lemma}
\label{lem:duality}
With the notation in Theorem~\ref{thm:top_trans}, for each $v \in SX$ there exists a sequence $\psi_n \in \Gamma$ such that for any $x \in X$, $\psi_n(x) \rightarrow v(\infty)$ and $\psi_n^{-1}(x) \rightarrow v(-\infty)$.
\end{lemma}

For a complete $\text{CAT}(0)$ space, this is the so-called \emph{duality condition}, see for instance~\cite[Chapter III]{Ball95}. Notice that by Lemma~\ref{lem:bd_dist} it is enough to show the convergence in the lemma above for a single $x \in X$.

\begin{proof}
As $SM$ has finite volume and the geodesic flow $g^t$ preserves the Liouville measure, Poincar\'{e} recurrence implies for almost every $v \in SM$ there exists $t_n \rightarrow \infty$ such that $g^{t_n}(v) \rightarrow v$. So for a dense set of $v \in SX$ there exists $\psi_n \in \Gamma$ and $t_n \rightarrow \infty$ such that $\psi_n(g^{t_n} v) \rightarrow v$. Then 
\begin{align*}
0=\lim_{n \rightarrow \infty} d\big(\psi_n(\gamma_v(t_n)),\gamma_v(0)\big) = \lim_{n \rightarrow \infty} d\big(\gamma_v(t_n),\psi_n^{-1}(\gamma_v(0))\big).
\end{align*}
Let $p = \gamma_v(0)$ then 
\begin{align*}
\lim_{n \rightarrow \infty} b_{\psi_n^{-1}(p)}(x) 
&= \lim_{n \rightarrow \infty} d(x,\psi_n^{-1}(p))-d(\psi_n^{-1}(p),o) \\
&= \lim_{n \rightarrow \infty} d(x,\gamma_v(t_n))-d(\gamma_v(t_n),o) \\
&= b_v(x)-b_v(o) = v(+\infty)
\end{align*}
so $\psi_n^{-1}(p) \rightarrow v(+\infty)$.

Now let $\gamma_n:\RR \rightarrow X$ be the geodesic defined by $\gamma_n(t) = \psi_n(\gamma_v(t_n+t))$. Then $v_n=\gamma_n^\prime(0) \rightarrow v$. We claim that $\psi_n(\gamma_n(-t_n)) \rightarrow v(-\infty)$. As $v_n \rightarrow v$ and $b_w^t(x) \rightarrow b_w(x)$ locally uniformly in $x \in X$ and $w \in SX$ we see that 
\begin{align*}
b_{v_n}^{-t_n}-b_{v_n}^{-t_n}(o)=b_{-v_n}^{t_n}-b_{-v_n}^{t_n}(o) \rightarrow b_{-v}-b_{-v}(o)=v(-\infty).
\end{align*}
So $\psi_n(\gamma_v(0)) = \psi_n(\gamma_n(-t_n)) \rightarrow v(-\infty)$.

Notice that the conclusion of the lemma is a closed condition in $SX$ and the argument above shows that it holds on a dense subset, so the lemma follows.
\end{proof}

The remainder of the proof closely follows Ballmann's book~\cite{Ball95}.

\begin{lemma} 
With the notation in Theorem~\ref{thm:top_trans}, suppose $v,w \in  SX$. Let $\psi_n$ be a sequence of isometries of $X$ such that $\psi_n^{-1}(x) \rightarrow v(\infty)$ and $\psi_n(x) \rightarrow w(-\infty)$. Then there are sequences $v_n \in SX$ and $t_n \in \RR$ such that $v_n \rightarrow v$ and $\psi_n(g^{t_n} v_n) \rightarrow w$.
\end{lemma}

\begin{proof}
For $m \geq 0$ and $n \in \NN$ let $\gamma_{m,n}:\RR \rightarrow X$ be a unit speed geodesic with 
\begin{align*}
\gamma_{m,n}(-m) = \gamma_v(-m) \text{ and } \gamma_{m,n}(t_{m,n}+m) = \psi_n^{-1}(\gamma_w(m))
\end{align*}
where
\begin{align*}
t_{m,n} = d(\gamma_v(-m),\psi_n^{-1}(\gamma_w(m)))-2m.
\end{align*}
Fix $m$, then Lemma~\ref{lem:bd_dist} implies that $\psi_n^{-1}(\gamma_w(m)) \rightarrow v(+\infty)$ and so Lemma~\ref{lem:convergence} implies that
\begin{align*}
\gamma_{m,n}^\prime(-m) \rightarrow \gamma_v^\prime(-m)
\end{align*}
as $n \rightarrow \infty$. So for $n$ large enough
\begin{align}
\label{eq:one}
d(\gamma_{m,n}(t),\gamma_v(t)) \leq 1/m \text{ for $-m \leq t \leq m$}
\end{align}
In a similiar fashion, for $n$ large enough
\begin{align}
\label{eq:two}
d(\psi_n \gamma_{m,n}(t_{m,n}-t),\gamma_w(-t)) \leq 1/m \text{ for $-m \leq t \leq m$}
\end{align}
Now pick $n=n(m)$ large enough so that~\ref{eq:one} and~\ref{eq:two} hold and let $v_m = \gamma_{m,n(m)}^\prime(0)$. Then $v_m \rightarrow v$ and $\psi_{n(m)} g^{t_{m,n(m)}} v_m \rightarrow w$.
\end{proof}

\begin{corollary} 
\label{cor:final}
With the notation above, let $v,w \in SX$ be unit speed geodesics in $X$ with $v(\infty)=w(\infty)$. Then there are sequences $v_n \in SX$, $t_n \in \RR$, and $\psi_n\in\Gamma$ such that $v_n \rightarrow v$ and $\psi_n(g^{t_n} v_n) \rightarrow w$.
\end{corollary}

\begin{proof}
By Lemma~\ref{lem:duality}, there is a sequence $\varphi_n \in \Gamma$ such that $\varphi_n(x) \rightarrow w(\infty)$ and $\varphi_n^{-1}(x) \rightarrow w(-\infty)$. Now as $v(\infty)=w(\infty)$ we can apply the previous Lemma.
\end{proof}

\begin{proof}[Proof of Theorem~\ref{thm:top_trans}]
If the geodesic flow is topologically transitive on $SM$, then there exists $v \in SM$ such the the geodesic $\gamma_v$ has dense image in $SM$. Now lift $v$ to $\wt{v} \in SX$ and consider $w \in SX$, then because $\gamma_v$ is dense in $SM$ there exists $t_n \in \RR$ and $\psi_n \in \Gamma$ such that $\psi_n \left(g^{t_n} \wt{v}\right) \rightarrow w$, which implies that $\psi_n \left(\wt{v}(\infty)\right) \rightarrow w(\infty)$. As $w \in SX$ was arbitrary, $\wt{v}(\infty)$ has a dense $\Gamma$-orbit in $\partial \hat{X}$.

Now suppose $\xi \in \partial \hat{X}$ has a dense $\Gamma$-orbit. For a set $W \subset SX$, define $W(\infty)$ to be the points $\eta \in \partial \hat{X}$ with $\eta = v(\infty)$ for some $v \in W$. To demonstrate that the geodesic flow is topologically transitive on $SM$ it is enough to show that for any open sets $U,V \subset SX$ there exists $\psi \in \Gamma$ and $t \in \RR$ such that $U \cap \psi(g^t V) \neq \emptyset$. Letting $U(\infty)$ and $V(\infty)$ be the endpoints at infinity, as $\xi$ has dense $\Gamma$-orbit there exists a $\psi \in \Gamma$ and $\eta \in \partial \hat{X}$ such that $ \eta \in U(\infty) \cap \psi V(\infty)$. Let $u \in U$ and $v \in V$ such that $u(\infty) = \psi \left(v(\infty)\right)$ and apply Corollary~\ref{cor:final} these geodesics.
\end{proof}

\section{The geodesic flow is topologically transitive}

The purpose of this section is to prove the following.

\begin{proposition}
If $M$ is a non-flat finite volume harmonic manifold without conjugate points, then the geodesic flow is topologically transitive on $SM$.
\end{proposition}

We will first show that there is a dense $\pi_1(M)$-orbit in the Busemann boundary of the universal cover of $M$ and then prove that the stable Riccati solution is continuous. By~\cite[Chapter 6]{besse78}, $M$ has bounded sectional curvature and then the proposition follows from Theorem~\ref{thm:top_trans}.

\begin{lemma}
\label{lem:ergodic}
Suppose $M$ is a finite volume non-flat harmonic manifold without conjugate points. Let $X$ be the universal Riemannian cover of $M$ with deck transformations $\Gamma=\pi_1(M) \subset \text{Isom}(X)$. If $\{ \nu_x : x\in X\}$ is the $\Gamma$-Patterson-Sullivan measure on $\partial \hat{X}$, then group $\Gamma$ acts ergodically on $(\partial \hat{X},\nu_o)$, in the sense that any $\Gamma$-invariant $\nu_o$-measurable set $A \subset \partial \hat{X}$ has either $\nu_o(A)=1$ or $\nu_o(A)=0$.
\end{lemma}

\begin{proof}
Assume $A$ is a $\Gamma$-invariant set with $\nu_o(A) >0$, then the function
\begin{align*}
H(x) = \frac{1}{\nu_o(A)}\int_{\partial \hat{X}} \chi_A(\xi) e^{-h_{vol}\xi(x)}d\nu_o(\xi)
\end{align*}
is harmonic and $\Gamma$-invariant. In particular it descends to a harmonic function on $M$, but $M$ has finite volume so $H$ must be constant. Now $H \equiv 1$ because 
\begin{align*}
H(o)= \frac{1}{\nu_o(A)}\int_{\partial \hat{X}} \chi_A(\xi) e^{-h_{vol}\xi(o)}d\nu_o(\xi)=\frac{1}{\nu_o(A)}\int_{\partial \hat{X} } \chi_A(\xi)d\nu_o(\xi)=1.
\end{align*}
Then identifying $\partial \hat{X}$ with the Poisson boundary of $X$ implies (by the uniqueness of the harmonic measure) that
\begin{align*}
d\nu_o(\xi) = \frac{\chi_A(\xi)d\nu_o(\xi)}{\nu_o(A)}.
\end{align*}
Thus $\nu_o(A)=1$.
\end{proof}

\begin{corollary}
Suppose $M$ is a finite volume non-flat harmonic manifold without conjugate points. Let $X$ be the universal Riemannian cover of $M$ with deck transformations $\Gamma=\pi_1(M) \subset \text{Isom}(X)$. Then there is a dense $\Gamma$-orbit in $\partial\hat{X}$.
\end{corollary}

\begin{proof}
By Proposition~\ref{prop:phi_map} and Theorem~\ref{thm:main}, $\partial \hat{X}$ is homeomorphic to $S_oX$ and under this identification $\nu_o$ is the Lebesque measure on $S_oX$. In particular, $\partial \hat{X}$ has a countable basis and every open set in $\partial \hat{X}$ has positive $\nu_o$-measure.

Then the corollary follows from Lemma~\ref{lem:ergodic} and a well known argument in ergodic theory: let $\{U_n\}$ be a countable basis of open sets for $\partial \hat{X}$. Then 
\begin{align*}
\xi \in \cap_{n \in \NN} \cup_{\gamma \in \Gamma} \gamma(U_n)
\end{align*}
if and only if the $\Gamma$-orbit of $\xi$ is dense. However for $n \in \NN$, the set $A_n=\cup_{\gamma \in \Gamma} \gamma(U_n)$ is $\Gamma$-invariant and has positive measure, hence $\nu_o(A_n)=1$. Thus $\nu_o$-almost every $\xi \in \partial \hat{X}$ has a dense $\Gamma$-orbit.
\end{proof}

\begin{lemma}
Let $X$ be a simply connected, non-compact harmonic manifold. Then the map $v \rightarrow U^s(v)$ is continuous.
\end{lemma}

\begin{proof}
It is well known that $\tr U^s(v) = h_{vol}$ for all $v \in SX$ (see for instance Heber~\cite[Remark 2.2]{heber06}). Then the map $v \rightarrow \tr U^s(v)$ is constant and therefore continuous. By Proposition~\ref{prop:riccati_basic}, $\tr U_T^s(v)$ converges to $\tr U^s(v)$ monotonically as $T \rightarrow \infty$. So Dini's theorem implies that $\tr U_T^s(v) \rightarrow \tr U^s(v)$ locally uniformly in $v \in SX$. 

We claim that $U^s_T(v) \rightarrow U^s(v)$ locally uniformly. Let $\lambda_{max}(v,T)$ be the maximum eigenvalue of $U^s(v)-U^s_T(v)$. As $U^s(v)-U^s_T(v)$ is positive semidefinite, $U^s_T(v)$ converges to $U^s(v)$ locally uniformly if and only if $\lambda_{max}(v,T)$ converges locally uniformly to zero. But
\begin{align*}
0 \leq \lambda_{max}(v,T) \leq \tr (U^s(v)-U^s_T(v))
\end{align*}
and so $U^s_T(v) \rightarrow U^s(v)$ locally uniformly. Then as $v \rightarrow U_T^s(v)$ is continuous the lemma follows.
\end{proof}

\appendix

\section{Remarks on convergence}

\subsection{Convergence of Busemann functions:}

The following is probably well known, but unable to find a reference we include the short proof.

\begin{lemma}
\label{lem:conv}
Suppose $(X,g)$ is a simply connected Riemannian manifold without conjugate points. If the sectional curvature of $(X,g)$ is bounded below, $y_n$ is a sequence in $X$, and $\xi \in \partial \hat{X}$ then the following are equivalent:
\begin{enumerate}
\item $b_{y_n}$ converges to $\xi$ locally uniformly on $X$,
\item $\nabla b_{y_n}$ converges to $\nabla \xi$ locally uniformly on $X$.
\end{enumerate}
\end{lemma}

\begin{proof}
Suppose $b_{y_n}$ converges to $\xi \in \partial \hat{X}$ locally uniformly on $X$. Since the curvature is bounded from below, by~\cite[Lemma 2.8]{eberlein73I} there exists $\kappa >0$ such that for $d(x,y_n)>1$ we have
\begin{align*}
\norm{\nabla^2 b_{y_n}(x) }=\norm{\nabla_x^2 d(x,y_n)} \leq \kappa.
\end{align*} 
Now the Arzel\`{a}-Ascoli theorem implies that $\nabla b_{y_n}$ converges to $\nabla \xi \in \partial \hat{X}$ locally uniformly on $X$.

The other direction is trivial.
\end{proof}  

\subsection{Convergence of harmonic functions:}

In this subsection we prove the following.

\begin{lemma}
\label{lem:harm_conv}
Suppose $(X,g)$ is a harmonic manifold. Let $\OO \subset X$ be an open subset, $h_n$ a sequence of harmonic functions on $\OO$, and $h$ a function on $\OO$. Then the following are equivalent:
\begin{enumerate}
\item $h_n \rightarrow h$ locally uniformly in $\OO$,
\item for all $k \geq 0$, $\nabla^k h_n \rightarrow \nabla^k h$ locally uniformly in $\OO$.
\end{enumerate}
\end{lemma}

This lemma follows from the basic theory of elliptic partial differential equations, but we will exploit the fact that $X$ is harmonic to provide a simple proof. 

For $C,\epsilon >0$, define $\phi_{\epsilon}: X \times X \rightarrow \RR$ by 
\begin{align*}
\phi_{\epsilon}(x,y) = \left\{
\begin{array}{ll}
C \exp\left( -(\epsilon^2-d(x,y)^2)^{-1} \right) & \text{for } d(x,y) \leq \epsilon \\
0 & \text{else.}
\end{array}\right.
\end{align*}
By construction $\phi_\epsilon \in C^{\infty}(X \times X)$ and for $y \in X$ fixed the function $x \rightarrow \phi_\epsilon(y,x)$ is compactly supported. Further as $X$ is harmonic the volume of a geodesic sphere only depends on its radius and we may pick $C=C(\epsilon)>0$ such that 
\begin{align*}
\int_X \phi_\epsilon(y,x)dy = 1
\end{align*}
for all $x \in X$. Now suppose $f: X \rightarrow \RR$ is continuous then basic analysis shows that the function
\begin{align*}
(\phi_\epsilon * f)(x) = \int_X f(y)\phi_\epsilon(y,x)dy
\end{align*}
is smooth and $\nabla^k(\phi_\epsilon * f)(x) = \int_X f(y)\nabla^k_x\phi_\epsilon(y,x)dy$. 

Let $h$ be a harmonic function on $\OO$. For $\epsilon$ positive and sufficiently small the function $\phi_{\epsilon} * h$ is defined on an open subset of $\OO$ and by Lemma~\ref{lem:MVT}, $\phi_\epsilon * h = h$ on this open subset. 

\begin{proof}[Proof of Lemma~\ref{lem:harm_conv}]
Suppose $h_n \rightarrow h$ locally uniformly in $\OO$. Fix a compact set $K \subset \OO$, then for $\epsilon$ positive and sufficiently small the set
\begin{align*}
B_{\epsilon}(K) = \{x \in X : d(x,K)\leq \epsilon\}
\end{align*}
is compact and contained in $\OO$. Then for $x \in K$ and $k \geq 1$:
\begin{align*}
\nabla^k h_n(x) = \nabla^k (\phi_\epsilon * h_n)(x) = \int_X h(y) \nabla^k_x \phi_\epsilon(y,x) dy
\end{align*}
and $\nabla^k h_n$ converges uniformly on $K$. As $K$ was an arbitrary compact subset of $\OO$, this shows that $\nabla^k h_n \rightarrow \nabla^k h$ locally uniformly in $\OO$.
\end{proof}

\bibliographystyle{alpha}
\bibliography{geom}

\end{document}